\documentclass[11pt]{amsart}
\usepackage{epsfig,color}
\usepackage{blindtext}
\usepackage{graphicx}
\usepackage{enumitem}
\usepackage{url}
\usepackage{amssymb}
\usepackage{graphicx,import}
\usepackage{comment}
\usepackage{esint}
\usepackage{xypic}

\usepackage[margin = 1.4in] {geometry}

\usepackage{hyperref}

\theoremstyle{theorem}

\newtheorem{theorem}{Theorem}[section]

\newtheorem{proposition}[theorem]{Proposition}
\newtheorem{lemma}[theorem]{Lemma}
\newtheorem{corollary}[theorem]{Corollary}

\theoremstyle{definition}
\newtheorem{example}[theorem]{Example}
\newtheorem{definition}[theorem]{Definition}
\newtheorem*{remark*}{Remark}

\newtheorem{remark}[theorem]{Remark}

\numberwithin{equation}{section}

\newcommand{\cA}{\mathcal{A}}
\renewcommand{\cH}{\mathcal{H}}
\newcommand{\R}{\mathbf{R}}

\newcommand{\N}{\mathbb{N}}
\newcommand{\Z}{\mathbb{Z}}

\DeclareMathOperator{\Ric}{Ric}
\DeclareMathOperator{\biRic}{biRic}
\DeclareMathOperator{\Lip}{Lip}
\DeclareMathOperator{\diam}{diam}
\newcommand{\eps}{\varepsilon}
\DeclareMathOperator{\Image}{image}
\DeclareMathOperator{\codim}{codim}
\DeclareMathOperator{\inj}{inj}

\title[Stable minimal hypersurfaces in 4-manifolds]{Complete stable minimal hypersurfaces in positively curved 4-manifolds}

\author{Otis Chodosh}
\address{Department of Mathematics, Stanford University, Building 380, Stanford, CA 94305, USA}
\email{ochodosh@stanford.edu}

\author{Chao Li}
\address{Courant Institute, New York University, 251 Mercer St, New York, NY 10012, USA}
\email{chaoli@nyu.edu}

\author{Douglas Stryker}
\address{Department of Mathematics, Princeton University, Fine Hall, 304 Washington Road, Princeton, NJ 08540, USA}
\email{dstryker@princeton.edu}

\begin{document}

\maketitle

% ABSTRACT
\begin{abstract}
We show that the \emph{combination} of non-negative sectional curvature (or $2$-intermediate Ricci curvature) and strict positivity of scalar curvature forces rigidity of complete (non-compact) two-sided stable minimal hypersurfaces in a $4$-manifold with bounded curvature. 

Our work leads to new comparison results. We also construct various examples showing rigidity of stable minimal hypersurfaces can fail under other curvature conditions. 
\end{abstract}

\section{Introduction}

Recall that a \emph{two-sided stable minimal hypersurface} $M^{n-1}\to (X^{n},g)$ is an immersed hypersurface with vanishing mean curvature and trivial normal bundle satisfying the inequality
\[
\int_{M} (|A_{M}|^{2} + \Ric_{g}(\nu,\nu))\varphi^{2} \leq \int_{M} |\nabla \varphi|^{2}
\]
for any $\varphi \in C^{\infty}_{c}(M\setminus\partial M)$. Here, $A_{M}$ is the second fundamental form of the immersion, $\Ric_{g}$ is the ambient Ricci curvature, and $\nu$ is (any) choice of unit normal. Stable minimal hypersurfaces can be used in a similar manner to stable geodesics to probe the ambient geometry of $(X^{n},g)$. The basic results along these lines are as follows. When $M^{n-1}\to (X^{n},g)$ is a \emph{closed} (compact without boundary) two-sided stable minimal hypersurface, it holds that
\begin{enumerate}
\item If $\Ric_{g}\geq 0$ then $M$ is totally geodesic and $\Ric_{g}(\nu,\nu) \equiv 0$ along $M$ \cite{Simons} (cf.\ \cite{BT}). In particular, when $\Ric>0$, there are no closed two-sided stable minimal hypersurfaces. 
\item If $R_{g}\geq 1$ (scalar curvature) then $M$ also admits a metric of positive scalar curvature \cite{SY:3d-torus,SY:descent}. In particular, when $n=3$ and $X$ is oriented, each component of $M$ must be a $2$-sphere. 
\end{enumerate}
The second result is a fundamental tool in the study of manifolds of positive scalar curvature. 

When $M$ is now assumed to be complete with respect to the induced metric (in particular, $M$ has no boundary) and \emph{non-compact}, the theory becomes considerably more complicated. As the following results show, it has been well-developed in $3$-dimensions. Consider $M^{2}\to (X^{3},g)$ complete two-sided stable minimal immersion:
\begin{enumerate}
\item  When $R_{g}\geq 0$, the induced metric on $M$ is conformal to either the plane or the cylinder \cite{fischer-colbrie-schoen}. In the latter case, $M$ is totally geodesic, intrinsically flat, and $R_{g}\equiv 0,\Ric_{g}(\nu,\nu) \equiv 0$ along $M$ (cf.\ \cite[Proposition C.1]{CCE}). 
\item When $\Ric_{g}\geq 0$, $M$ is totally geodesic, intrinsically flat, and $\Ric_{g}(\nu,\nu) \equiv 0$ along $M$ \cite{SY:ric} (cf.\ \cite{fischer-colbrie-schoen,docarmo-peng,pogorelov}).
\item When $R_{g}\geq 1$, $M$ must be compact \cite{SY:condensation} (cf.\ \cite{GL:complete}). 
%\item  
\end{enumerate}
These results have had important applications to comparison geometry. In particular, we refer to the following (incomplete) list of results that rely specifically on (non-compact) stable minimal surfaces in $3$-manifolds: \cite{SY:ric,GL:complete,AR,Liu,Wang:whitehead}. For most of these applications, it is essential that no assumption is made related to properness or volume growth of $M$.

%In higher dimensions, there are several important results concerning complete (non-compact) two-sided stable minimal immersions $M^{n-1}\to(X^{n},g)$ (e.g., \cite{SSY,CaoShenZhu,LiWang:ends.stable.minimal}), but these results either place additional topological or geometric assumptions (such as volume growth) on $M$ or only have coarse topological conclusions (such as one-endedness).

The first- and second- named authors have recently resolved \cite{CLstable} the conjecture of Schoen that a complete two-sided stable minimal immersion $M^{3}\to\mathbf{R}^{4}$ is flat (without any additional assumptions on $M$); see also the subsequent proofs in \cite{ChodoshLi:aniso,CatinoMastroliaRoncoroni}. It is natural to ask what happens in a curved ambient background, rather than flat Euclidean space. For example, a basic question is to determine which ambient curvature conditions suffice to show that a complete stable minimal hypersurface must be compact, or cannot exist. The methods used in \cite{CLstable} do not seem to extend to the curved setting.\footnote{In particular, the use of the $L^3$ Schoen--Simon--Yau inequality \cite{SSY} seems troublesome in most curved ambient manifolds.} (See \cite{SSY,ShenYe,CaoShenZhu,LiWang:ends.stable.minimal} for previous results in this direction.)

In fact, (as was discussed above) in a $3$-manifold, $R_{g}\geq 1$ implies compactness and $\Ric_{g}>0$ implies non-existence; on the other hand, in $4$-dimensions, neither of these results can hold:
\begin{example}[Non-compact stable minimal hypersurface in $4$-manifold with $R_{g}\geq 1$]\label{exam:psc}
Fix an oriented closed $2$-dimensional Riemannian manifold $(X^{2},g)$ admitting 
\[
\sigma : \mathbf{R} \to (X^{2},g)
\]
unit speed stable geodesic. Then, for $\eps>0$ sufficiently small $(X^{2},g) \times \mathbf{S}^{2}(\eps)$ has $R_{g}\geq 1$, while $\sigma \times \mathbf{S}^{2}(\eps)$ is an unbounded two-sided stable minimal immersion. 
\end{example}
\begin{example}[Stable minimal hypersurface in a $4$-manifold with positive sectional curvature]\label{exam:hardy}
Consider a rotationally symmetric metric $g$ on $\mathbf{R}^{4}$ with strictly positive sectional curvatures. There are totally geodesic copies of $\mathbf{R}^{3}$ in such a metric, and as long as the Ricci curvature of $g$ decays to zero sufficiently fast, these will be (complete) two-sided stable minimal embeddings. Essentially, the point is that on flat Euclidean $3$-space $(\mathbf{R}^{3},\bar g)$, any Schr{\"o}dinger operator with non-negative, rapidly decaying potential\footnote{Note that a Schr{\"o}dinger operator with non-negative (but not identically zero) potential cannot be stable on $\mathbf{R}^{2}$, thanks to the log-cutoff trick.} will be stable, thanks to a classical Hardy inequality
\[
\int_{\mathbf{R}^{3}} 4|\bar\nabla \varphi|^{2}d\bar\mu \geq \int_{\mathbf{R}^{3}} r^{-2} \varphi^{2}d\bar\mu. 
\]
We explicitly construct such an example and check its stability in Section \ref{sec:pos.K.cx}.
\end{example}

\subsection{Main results}  It turns out that if one \emph{combines} positivity of curvature with strict positivity of scalar curvature, such examples can be avoided. This is the main result of this article. 

\begin{theorem}\label{thm:main.sec}
If $(X^{4},g)$ is a complete $4$-manifold with weakly bounded geometry, non-negative sectional curvature, and strictly positive scalar curvature $R_g\geq R_{0}>0$, then any complete two-sided stable minimal immersion $M^{3}\to(X^{4},g)$ is totally geodesic and has $\Ric_{g}(\nu,\nu) \equiv 0$ along $M$. 
\end{theorem}

Here, weakly bounded geometry means that around any point, there is a \emph{local} diffeomorphism from a ball in $\R^{4}$ of definite size so that the pullback metric is $C^{1,\alpha}$-close to Euclidean (see Definition \ref{defi:bound.geo}). (Note that this allows for collapsing behavior at infinity.)
\begin{remark}
All results in this paper assume that $(X,g)$ at least has $\Ric_g\geq 0$, in which case ``weakly bounded geometry'' could be replaced by the assumption ``$\sec_g\leq K$'' anywhere it occurs here. It would be interesting to understand if this condition could be removed. 
\end{remark}

\begin{remark}
Shen--Ye \cite{ShenYe} proved that for $n\leq 4$, if $(X^{n+1},g)$ satisfies
\begin{equation}\label{eq.shenye.curvature}
	\biRic \ge k>0,
\end{equation}
then any complete two-sided stable minimal immersion $M^n \to (X^{n+1},g)$ is compact. Here $\biRic$ is called the bi-Ricci curvature, defined in \cite{ShenYe} ($\biRic\ge k$ implies that $R_g\ge c_n k$ for some dimensional constant $c_n$). (See Appendix \ref{sec:curv-cond} for precise definitions of these curvature conditions as well as Section \ref{sec:drop-sectional} below.) 

In particular, it follows that such a minimal immersion does not exist if $(X^{n+1},g)$ is assumed to be closed with positive sectional curvature. After our paper appeared, the Shen--Ye result was rediscovered and extended to $n=5$ by Catino--Mastrolia--Roncoroni \cite{CatinoMastroliaRoncoroni}, under the additional assumption of nonnegative sectional curvature (or even just nonnegative 4-intermediate Ricci curvature). The approach in \cite{ShenYe,CatinoMastroliaRoncoroni} is very different from that of this paper and is based on a conformal deformation technique first introduced by Fischer-Colbrie \cite{FC}. An interesting feature of our paper (see Theorem \ref{thm:main}) is that we are able to replace (when $n=3$) the condition $\biRic\ge k>0$ by the weaker assumption on scalar curvature $R_g\ge R_0>0$ in \eqref{eq.shenye.curvature}. We note that the condition considered in this paper allows for the existence of complete, \textit{non-compact} stable minimal immersions (we prove they must be totally geodesic and have vanishing normal Ricci curvature). For example, one may consider $S^2\times \mathbf{R} \to S^2 \times \mathbf{R} \times S^1$. 
\end{remark}

Of course, any closed manifold $(X,g)$ has weakly bounded geometry, so we have the following result:

\begin{corollary}\label{cor:cpt.sec}
If $(X^{4},g)$ is a closed $4$-manifold with non-negative sectional curvature and positive scalar curvature, then a complete two-sided stable minimal immersion $M^{3}\to(X^{4},g)$ is totally geodesic and $\Ric_g(\nu,\nu) \equiv 0$ along $M$. 
\end{corollary}

This applies (for example) to the standard product metrics on $S^4$, $S^1 \times S^3$, $S^2\times S^2$, and $T^2 \times S^2$. We emphasize that the latter example does not have positive bi-Ricci curvature, and is thus not covered by the results in \cite{ShenYe,CatinoMastroliaRoncoroni}. 

\begin{remark}
The study of complete (non-compact) stable minimal immersions in a \emph{compact} $3$-manifold has played an important role in the study of embedded minimal surfaces with bounded Morse index \cite{CKM,Carlotto:bd.morse,LiZhou}. One may hope that the Corollary \ref{cor:cpt.sec} could lead to similar results in closed $4$-manifolds. 
\end{remark}

\begin{remark}
Theorem \ref{thm:main.sec} fails in $6$-dimensional manifolds (and higher). See Appendix \ref{sec:six-dim} for a counterexample. Gromov's conjecture that a $4$-manifold with uniformly positive scalar curvature has macroscopic dimension $2$ suggests that the $5$-dimensional version of Theorem \ref{thm:main.sec} would be the ``critical'' dimension (this is related to the log-cutoff trick for minimal surfaces in $3$-manifolds). 
\end{remark}

\subsection{Concerning the assumption of non-negative sectional curvature}\label{sec:drop-sectional} It is natural to ask whether or not there exist complete two-sided stable minimal immersions $M^{3}\to (X^{4},g)$ where $(X^{4},g)$ has weakly bounded geometry, $\Ric_{g}\geq 0$ and $R_{g}\geq 1$. (Note that in Theorem \ref{thm:main.sec}, $\Ric_{g}\geq 0$ is replaced by non-negative sectional curvature.) 

To this end, we have the following example indicating that this is unlikely to be true. 

\begin{example}[Stable minimal hypersurface in a $4$-manifold with strictly positive Ricci curvature]\label{exam:Ricci}
There exists a closed 4-manifold $(X^{4},g)$ (so it automatically has bounded geometry) and a complete two-sided stable minimal immersion $M^{3} \to (X^{4},g)$ so that $\Ric_{g}\geq 1$ (so, in particular $R_{g}\geq 1$) in some $\delta$-neighborhood of the image of $M$. 

This example is constructed in detail in Section \ref{sec:pos.norm.Ric}, but we briefly describe it here. The minimal immersion $M^{3} \to (X^{4},g)$ is constructed by taking the universal cover of some embedded compact\footnote{Note that $M_{0}$ is necessarily unstable, since $\Ric_{g}(\nu,\nu) > 0$; the instability occurs due to the ``largeness'' of the universal cover.} minimal submanfold $M_{0}\subset (X^{4},g)$. We choose $M_{0}$ so that it is diffeomorphic to $(S^{1}\times S^{2}) \# (S^{1}\times S^{2})$, so $\pi_{1}(M_{0}) = F_{2}$ is a ``non-amenable'' group. This implies that the universal cover has positive first eigenvalue. As such, as long as $|A_{M}|^{2} + \Ric_{g}(\nu,\nu)$ is sufficiently small, the universal cover will be stable. 

We note that it is unclear if one can construct such an example where $(X^{4},g)$ is complete and satisfies $\Ric_{g}>0$ and $R_{g}\geq 1$ at all points (not just near the image of $M$). (We discuss the possibility of doing so further in Section \ref{sec:pos.norm.Ric}.) However, we emphasize that this example precludes any argument concerning stable minimal immersions in such $(X^{4},g)$ that is based purely on the second variation of area.\footnote{Note that Li--Wang have used the Busemann function in their analysis of \emph{proper} ends of stable minimal hypersurfaces in non-negatively curved manifolds \cite{LiWang:ends.stable.minimal}. As such, this argument relies on the ambient geometry outside of just a tubular neighborhood of the immersion. However, it is unclear how to extend such a argument to non-negative Ricci curvature (the Busemann function is subharmonic rather than convex so it is difficult to control the restriction to a minimal hypersurface). Along these lines, we note that the stable minimal immersion $M^3\to (X^{4},g)$ discussed in Example \ref{exam:Ricci} has infinitely many ends. }
\end{example}

On the other hand, it is possible to weaken non-negativity of sectional curvature to an intermediate condition. Namely, we say that $(X^{4},g)$ has non-negative $2$-intermediate Ricci curvature if the sectional curvatures satisfy $\text{sec}_g(\Pi_{1}) + \text{sec}_g(\Pi_{2})\geq 0$ for any two $2$-planes $\Pi_{1},\Pi_{2}\subset T_{p}X$ intersecting in a line at a right angle (cf.\ Definition \ref{defi:2-int-Ric}.) We will write $\Ric_{2}^g\geq 0$ for this condition. (Note that non-negative sectional curvature implies $\Ric_{2}^g\geq 0$ which implies $\Ric_g\geq 0$.) The results discussed above all hold with non-negative sectional curvature replaced by $\Ric_{2}^g\geq 0$. More precisely, we have

\begin{theorem}\label{thm:main}
If $(X^4, g)$ is a complete 4-manifold with weakly bounded geometry, $\Ric_{2}^g\geq 0$, and $R_g \geq R_{0}>0$, then any complete two-sided stable minimal immersion $M^3 \to (X^4, g)$ is totally geodesic and has $\mathrm{Ric}_g(\nu, \nu) \equiv 0$ along $M$.
\end{theorem}

We discuss the topological implications of this result in Section \ref{sec:topo}.

We remark that $\Ric_g\geq 0$ (instead of $\Ric_{2}^g\geq 0$) suffices for most (but not all) steps of the proof of Theorem \ref{thm:main}, so if one places additional topological assumptions on the minimal hypersurface, then the proof of Theorem \ref{thm:main} can be adapted in a straightforward way to show the following result. 

\begin{theorem}
Let $(X^{4},g)$ be a complete $4$-manifold with weakly bounded geometry, $\Ric_{2}^g\geq -K$, $\Ric_g\geq 0$, and $R_g\geq R_{0}>0$. Let $M^{3}\to(X^{4},g)$ be a complete two-sided stable minimal immersion with finitely many ends and $b_{1}(M) < \infty$. Then $M$ is totally geodesic and has $\Ric_{g}(\nu,\nu) \equiv 0$ along $M$.
\end{theorem}

\subsection{Topological applications} 
Recall that Schoen--Yau used minimal surfaces to prove that a complete non-compact $(X^{3},g)$ with $\Ric_g>0$ is diffeomorphic to $\R^{3}$ \cite{SY:ric} (cf.\ \cite{AndersonRodriguez,Liu}). They did this by showing that $\pi_{2}(X) = 0$ and that $X$ is simply connected at infinity. In this section, we observe that given Theorem \ref{thm:main}, similar techniques can be used to prove new topological restrictions on $(X^{4},g)$ with $\Ric_{2}^g>0$, $R_g\geq R_{0}>0$, and weakly bounded geometry. (See \cite{Mouille:web} for a collection of results concerning the $\Ric_{k}^g$ curvature condition.)

All homology groups and cohomology groups here will be taken with $\Z$-coefficients.
\begin{definition}
	Let $X$ be a non-compact manifold, $k\in \mathbb{N}$. Define the $k$-th homology group of $X$ at infinity as the inverse limit
	\[H_k^\infty(X) = \varprojlim_{A\Subset X} H_k(X-A).\]
	We say $X$ is $H_k$-trivial at infinity if the natural homeomorphism $H_k^\infty(X)\to H_k(X)$ is injective.
\end{definition}

In section \ref{sec:topo} we prove the following result. 

\begin{theorem}\label{thm:topo}
	Suppose $(X^4,g)$ is a complete, non-compact, oriented $4$-manifold with weakly bounded geometry, $\Ric_2^g > 0$, and strictly positive scalar curvature $R_g\ge R_0>0$. Then we have:
	\begin{enumerate}
		\item $H_1(X)$ only has torsion elements;
		\item $X$ is $H_2$-trivial at infinity. 
	\end{enumerate}
\end{theorem}

It would be interesting to see if the techniques from \cite{AndersonRodriguez,Liu} (cf.\ \cite{CCE,CEM}) could be adapted to study the $\Ric_{2}^g\geq 0$ rigidity version of this result. 

\begin{remark}
If $(X^{n},g)$ has $\Ric_g > 0$, then $H_{n-1}(X) = 0$ by \cite{Yau:function.theoretic.complete,ShenSormani}. 
\end{remark}

In light of Examples \ref{exam:psc} and \ref{exam:hardy}, it would be interesting to know if Theorem \ref{thm:topo} held under weaker curvature conditions (e.g., $\Ric_{2}^g>0$ but without $R_g\geq R_{0}>0$).

\subsection{Idea of the proof and related results} An interesting feature of Theorem \ref{thm:main.sec} is that it combines different curvature conditions. Indeed, the different conditions come in at rather different places in the proof. As a consequence, the strategy of the proof is quite different from most previous results concerning stability of complete (non-compact) minimal surfaces (including the previous work of the first- and second-named authors on stable minimal immersions $M^3\to\mathbf{R}^4$ \cite{CLstable}). A key feature of this paper is the use of $\mu$-bubbles, a technique introduced by Gromov \cite{gromov1996positive}, allowing one to gain distance control (in certain settings) from positivity of scalar curvature. 

The overall strategy is as follows, indicating which curvature conditions enter at which step. Consider $M^3\to (X^{4},g)$ a complete two-sided stable minimal hypersurface. 
\begin{enumerate}
\item To show that $M$ is totally geodesic and has vanishing normal Ricci curvature, we would like to find a sequence of compactly supported functions $\varphi_{i}$ with $\varphi_{i}\to 1$ and $\int_{M}|\nabla\varphi_{i}|^{2} = o(1)$. Taking such functions in the stability inequality then gives the desired conclusion (this just uses $\Ric_g\geq 0$). Note that for minimal surfaces in a $3$-manifold, this condition is usually guaranteed by proving that the surface is conformal to the plane and then using the log-cutoff trick. However, in the higher dimensions, this is not usually not possible (e.g., no such functions exist on flat $\mathbf{R}^{3}$).  
\item The non-negativity of sectional curvature (or $\Ric_{2}^g\geq 0$) ensures that $M$ has at most one nonparabolic end (see Definition \ref{defi:parabolic}). Nonparabolic ends are more difficult to handle, since parabolic ends automatically admit test good functions as described in (1). 
\item The remaining issue\footnote{The reality is somewhat more complicated, since the nonparabolic component could have infinitely many parabolic ends attached, but this captures the main idea.} is to construct a good test function along the nonparabolic end. In general, this is not possible with only a non-negative sectional curvature assumption (cf.\ Example \ref{exam:hardy}). This is where the positivity of scalar curvature comes in. One knows (going back to the work of Schoen--Yau \cite{SY:descent}) that a stable minimal hypersurface in a positive scalar curvature manifold ``acts'' as if it itself has positive scalar curvature. 

In particular, a $3$-dimensional manifold with positive scalar curvature tends to be macroscopically $1$-dimensional in various senses (cf.\ \cite{GL:complete,MarquesNeves:width-psc,Song:embed,LiokumovichMaximo2020waist}). As such, one might imagine that a single end of $M$ has \emph{linear} volume growth (this is where the parabolic/nonparabolic distinction from (2) is important; a metric tree is $1$-dimensional but has exponential length growth, compare with the universal cover of $(S^{1}\times S^{1}) \# (S^{1}\times S^{2})$). Note that if this end had linear volume growth, then the standard cutoff function that is $1$ in $B_{R}$ and $0$ in $B_{2R}$ would have vanishing Dirichlet energy as $R\to\infty$, so this would allow us to construct the desired cutoff function.

It seems difficult (or potentially impossible)\footnote{One can imagine that $E$ is a cylinder $S^{2}\times (0,\infty)$ connected sum at integer points with spheres $S^{3}$ with volume tending to $\infty$ and scalar curvature $\geq 1$. It is unclear if this could be ruled out by stability. }  to actually prove that the nonparabolic end $E$ has linear volume growth. Instead we construct an exhaustion of $E$ by bounded sets $\Omega_{1}
\subset \Omega_{2}\subset \dots$ so that $\partial\Omega_{i}\cap E$ has controlled diameter. To do so, we use the theory of $\mu$-bubbles (surfaces of carefully prescribed mean curvature) first introduced by Gromov \cite{gromov1996positive} (see also \cite{Gromov:metric-inequalities}), allowing one to study metric quantities related to scalar curvature. (This application is similar in spirit to the use of $\mu$-bubbles in \cite{CL:aspherical,Gromov2020metrics,CLL:suff.conn}.) This is the step that uses the assumption of strictly positive scalar curvature. 

\item Finally, to construct a good test function, it remains to show that the $L$-neighborhood of $\partial \Omega_{i} \cap E$ has bounded volume (for some constant $L$). To upgrade diameter control (obtained in (3)) to volume control, we need to control (from below) the intrinsic Ricci curvature of the minimal immersion $M$. This is achieved (thanks to the Gauss equations) by combining a lower bound on the sectional (or $\Ric_{2}^g$) curvatures with an absolute bound on the second fundamental form of $M$ as follows from \cite{CLstable}. This is where the weakly bounded geometry assumption enters. 
\end{enumerate}

We remark that steps (3) and (4) are somewhat reminiscent of a recent result (but not the proof) of Munteanu--Wang proving that a complete $3$-manifold with $R_{g}\geq 1$ and $\Ric_g\geq 0$ has linear volume growth \cite{MW:2} (see also \cite{Xu:scal.int,MW:1,Zhu:psc.int}). See our recent work \cite{CLS:volumePSC} for an approach to this result following the methods of this paper.

\subsection{Acknowledgements} We are grateful to Richard Bamler, Renato Bettiol, Neil Katz, Dan Lee, Christina Sormani, and Ruobing Zhang for some helpful discussions. O.C.\ was supported by an NSF grant (DMS-2016403), a Terman Fellowship, and a Sloan Fellowship. C.L.\ was supported by an NSF grant (DMS-2005287). D.S.\ was supported by an NDSEG Fellowship.

\subsection{Organization of the paper}

In Section \ref{sec:curv} we discuss the $\Ric_{2}\geq 0$ and bounded geometry conditions used here. Section \ref{sec:ends.prelim} contains some preliminary results concerning the ends of non-compact manifolds. We then study the nonparabolic ends in further detail in Section \ref{sec:non.par.stab}. We discuss the $\mu$-bubble exhaustion result and corresponding volume growth estimates and then prove the main results in Section \ref{sec:proofs.main}. We discuss some topological applications of the main results in Section \ref{sec:topo}.  Appendix \ref{sec:curv-cond} contains various curvature conditions referred to in the paper. Appendix \ref{sec:exam} contains some examples relevant to the main results. Finally, Appendix \ref{app:pullback} describes how to pullback immersions along local diffeomorphisms. Finally, Appendix \ref{app:bd.geo} relates sectional curvature bounds to the weakly bounded geometry condition. 

\section{Curvature Preliminaries}\label{sec:curv}

If $(X, g)$ is a Riemannian manifold, we denote the Riemann curvature tensor by $R_g(\cdot, \cdot, \cdot,\cdot)$ with 4 arguments, the sectional curvature by $\text{sec}_g$, the Ricci curvature by $\Ric_g$, and the scalar curvature by $R_g$ with no argument. If $M \to (X, g)$ is an immersion, we write $\Ric_M$ and $R_M$ to denote respectively the Ricci curvature and scalar curvature of the pullback metric on $M$. We follow the curvature conventions used in \cite[Chapter 3]{Petersen:Riemannian}; in particular, if $\{u,v\}$ is an orthonormal basis for a $2$-plane $\Pi \subset T_p X$, then $\text{sec}_g(\Pi) = R_g(v,u,u,v)$.

\subsection{The $\Ric_{2}\geq0$ Condition}

Let $(X^4, g)$ be a complete Riemannian 4-manifold. The following curvature condition is tailored to exploit the Gauss equation for hypersurfaces in 4-manifolds, and sits between non-negative sectional curvature and non-negative Ricci curvature. 

\begin{definition}\label{defi:2-int-Ric}
$(X^4, g)$ satisfies $\Ric_{2}^g\geq 0$ if
\[ R_g(v, u, u, v) + R_g(w, u, u, w) \geq 0 \]
for any orthonormal vectors $\{u, v, w\}$ in $TX$.
\end{definition}
This condition has been called non-negativity of \emph{$2$-intermediate Ricci curvature} in the literature (cf.\ Appendix \ref{sec:curv-cond}). In particular, we note that there is a metric on $S^{2}\times S^{2}$ with strictly positive $2$-intermediate Ricci curvature \cite{Muter} (see \cite[Example 2.3]{Mouille:torus}). 

\begin{lemma}\label{lem:k2+}
Let $(X^4, g)$ be a 4-manifold with $\Ric_{2}^g\geq 0$. If $M^3 \to (X^4,g)$ is a minimal immersion, then
\[ \mathrm{Ric}_M \geq -|A_{M}|^{2}. \]
\end{lemma}
\begin{proof}
Pick an orthonormal basis $\{e_l\}$ at a point $p \in M$ with $e_4$ normal to $M$. The traced Gauss equation gives
\[ R_g(e_2, e_1, e_1, e_2) + R_g(e_3, e_1, e_1, e_3) - \mathrm{Ric}_M(e_1, e_1) = \sum_{j=1}^3 A_M(e_1, e_j)^2. \]
Rearranging and using $\Ric_{2}^g\geq0$ we have
\[ \mathrm{Ric}_M(e_1, e_1) \geq R_g(e_2, e_1, e_1, e_2) + R_g(e_3, e_1, e_1, e_3) - |A_{M}|^2 \geq -|A_{M}|^{2}. \]
Since $e_1$ was an arbitrary unit vector in $T_pM$, the conclusion follows.
\end{proof}

\subsection{Weakly Bounded Geometry}
\begin{definition}\label{defi:bound.geo}
	In this paper, we say that a complete Riemannian manifold $(X^{n}, g)$ has $Q$-\emph{weakly bounded geometry} if for any $p \in X$, there is a $C^{2,\alpha}$ local diffeomorphism $\Phi : ((B(0,Q^{-1}),0)\subset \mathbb{R}^{n} \to (U,p)\subset X$ so that 
	\begin{enumerate}
		\item we have $e^{-2Q} \delta \leq \Phi^{*}g \leq e^{2Q} \delta$ in the sense of bilinear forms, and
		\item it holds that $\Vert \partial_{k} \Phi^{*}g_{ij}\Vert_{C^{\alpha}}\leq Q$. 
	\end{enumerate}
\end{definition}
We will often say that $(X,g)$ has ``weakly bounded geometry'' if the above definition holds for some $Q$. 

Note that this condition allows for some collapsing behavior of $(X,g)$ at infinity.\footnote{For example, a hyperbolic cusp has weakly bounded geometry but the injectivity radius tends to zero at infinity.} It is well-known that $|\textrm{sec}_g|\leq K <\infty$ implies $Q$-weakly bounded geometry for $Q=Q(K)$ (cf.\ \cite{RST}). We outline the proof of this fact in Proposition \ref{prop:bd.geo}. Note also that $\Ric_g\geq 0$ and $\sec_g \leq K$ imply that $|\sec_g|\leq (n-2)K$, and in this paper the ambient space $(X^{4},g)$ is always assumed to have $\Ric_g\geq 0$ (often a stronger condition is assumed).

Alternatively, we recall that $|\textrm{Ric}_g|\leq K$ and $\inj(X, g) \geq i_{0}>0$ also implies this condition (and actually one can replace ``local diffeomorphism'' by ``diffeomorphism'' in this case); this follows from \cite{Anderson} (cf.\ \cite[Theorem 11.4.3]{Petersen:Riemannian}).

The first- and second-named authors have recently proven that a two-sided complete stable minimal hypersurface $M^{3}\to\mathbf{R}^{4}$ is flat \cite{CLstable}. Using a standard blow-up argument, this yields curvature estimates for stable minimal hypersurfaces in (compact) $4$-manifolds. Here, we observe that the same thing holds for non-compact $4$-manifolds, assuming the weakly bounded geometry condition. The proof below is similar to the argument used in \cite{RST}  (see also \cite{Cooper}) to prove curvature estimates for stable minimal (or more generally CMC) surfaces in $3$-manifolds (cf.\ \cite{Schoen:estimates}).

\begin{lemma}\label{lem:boundedgeometry}
	Let $(X^4, g)$ be a complete 4-manifold with $Q$-weakly bounded geometry. Then there is a constant $C=C(Q) < \infty$ such that every compact two-sided stable minimal immersion $M^3 \to (X^4, g)$ satisfies
	\[ \sup_{q\in M} |A_M(q)|\min\{1, d_M(q, \partial M)\} \leq C. \]
\end{lemma}
\begin{proof}
	This follows by combining a standard point picking argument (cf.\ \cite[Theorem 3]{CLstable}) with the lifting argument described in Appendix \ref{app:pullback}. For the sake of completeness we describe the full argument below. 
	
	For contradiction, we assume that the assertion fails. Then there is a sequence of compact two-sided stable minimal immersions $M_i^3 \to (X^4_i, g_i)$ such that
	\[
	\sup_{q\in M_{i}} |A_{M_i}(q)| \min\{1,d_{M_{i}}(q,\partial M_{i})\} \to \infty
	\] 
	as $i\to\infty$. Since $M_{i}$ is compact and the argument of the supremum is continuous and vanishes on $\partial M_{i}$, there is $p_{i}\in M_{i}\setminus\partial M_{i}$ with 
	\[
	|A_{M_i}(p_{i})| \min\{1,d_{M_{i}}(p_{i},\partial M_{i})\} = \sup_{q\in M_{i}} |A_{M_i}(q)| \min\{1,d_{M_{i}}(q,\partial M_{i})\} \to \infty. 
	\]
	Define $r_{i} = |A_{M_i}(p_{i})|^{-1} \to 0$ and $x_{i}$ the image of $p_{i}$ in $X_i$. By the weakly bounded geometry assumption, there are local diffeomorphisms 
	\[
	\Phi_{i} : (B(0,Q^{-1}),0)\subset \mathbf{R}^{4} \to (X_i,x_{i})
	\]
	with $e^{-2Q}\delta \leq \Phi^{*} g_i \leq e^{2Q}\delta$ and $\Vert \partial_{a}\Phi_{i}^{*}(g_i)_{bc}\Vert_{C^{\alpha}}\leq Q$. 
	
	By the pullback operation described in Appendix \ref{app:pullback}, we can find a sequence of pointed $3$-manifolds $(S_{i},s_{i})$, immersions $F_{i} : (S_{i},s_{i}) \to (B(0,Q^{-1}),0)$, and local diffeomorphisms $\Psi_{i} : (S_{i},s_{i})\to (M_i,p_{i})$ so that the following diagram commutes (writing $B=B(0,Q^{-1})\subset \mathbf{R}^{4}$)
	\[
	\xymatrix{
		S_{i} \ar[r]^{F_{i}} \ar[d]_{\Psi_{i}}& B\ar[d]^{\Phi_{i}}\\
		M_i \ar[r] & X_i
	}
	\]
	and so that $F_{i} : S_{i}\to (B,\Phi_{i}^{*}g_i)$ is a two-sided stable minimal immersion. (Two-sided and minimality follow immediately as they are local properties, while stability follows by lifting a positive first eigenfunction for the stability operator on $M_i$ to $S_{i}$.)
	
	Define maps $D_i : B(0,r_i^{-1}Q^{-1}) \to B(0,Q^{-1})$, $x\mapsto r_i x$ and then consider the metrics $\tilde g_i : = r_i^{-2} D_i^*\Phi^* g_i$ on $B(0,r_i^{-1}Q^{-1})$. The bounded geometry condition ensures that $\tilde g_i$ converges to the flat metric $\delta$ on $\mathbf{R}^4$ in $C^{1,\alpha}_{\textrm{loc}}$ (in sense of $C^{1,\alpha}_\textrm{loc}$ convergence of the metric coefficients in Euclidean coordinates). In particular, the Christoffel symbols of $\tilde g_i$ with respect to the Euclidean coordinates on $B(0,r_i^{-1}Q^{-1})$ converge to $0$ in $C^{0,\alpha}_\textrm{loc}$. 
	
	We now consider the rescaled immersions $\tilde F_i : = D_i^{-1} \circ F_i : S_i \to B(0,r_i^{-1}Q^{-1})$. Observe that $\tilde F_i$ is a minimal immersion with respect to $\tilde g_i$. Furthermore, by the point picking argument, if $q\in S_i$ has $d(s_i,q) \leq \rho$ (with respect to $\tilde F_i^* g_i$) then the second fundamental form of $\tilde F_i$ with respect to $\tilde g_i$ satisfies $|A_{\tilde F_i}(q)| \leq C=C(\rho)$. 
	
	The bounded curvature condition ensures there is $\mu>0$ (a numerical constant) so that for any $q \in S_{i}$, for $i$ sufficiently large, we can write $B_{\mu}^{S_{i}}(q)$ as a normal graph (in the Euclidean coordinates) over a subset of $T_{q}S_{i}$ of a function $f_i$ with (Euclidean) $C^{2}$-norm $\leq \mu$ (cf. \cite[Lemma 2.4]{CM:book}, except one should use that the Christoffel symbols are uniformly controlled). Geometric considerations show that the area of such a graph depends on the metric coefficients of $\tilde g_i$ (but not derivatives of the metric). In particular, the first variation formula (and minimality of $\tilde F_i$) implies that $f_i$ satisfies a (uniformly) elliptic PDE in divergence form whose coefficients depend on the metric coefficients $\tilde g_i$ (but not derivatives). Thus, $f_i$ satisfies a non-divergence form elliptic PDE, whose coefficients depend on the coefficients of $\tilde g_i$ and $\partial \tilde g_i$ (in Euclidean coordiantes). These coefficients are uniformly controlled in $C^{\alpha}$, and thus Schauder estimates yield interior $C^{2,\alpha}$-estimates for $f_i$. In particular, the injectivity radius of the pull back metric on $\tilde h_i : = \tilde F_i^* \tilde g_i$ is locally uniformly bounded. The Gauss equations imply that the sectional curvature of $\tilde h_i $ is locally uniformly bounded. 
	
	In particular, the injectivity radius and curvature bounds imply that we can take a subsequential $C^{1,\alpha}$ limit of $(S_i,\tilde h_i ,s_i)$ in the pointed Cheeger--Gromov sense (see \cite[Theorem 11.4.7]{Petersen:Riemannian}) to a limit $(S,h,s)$. Recall that this means that for any $s \in \Omega\Subset S$ there is $i_0(\Omega)$ large so that for $i\geq i_0(\Omega)$ there are embeddings
	\[
	G_i : (\Omega,s) \to (S_i,s_i) 
	\] 
	so that $G_i^* \tilde h_i \to h$ in the $C^{1,\alpha}$-topology. 
	
	We now verify that we can also pass the immersions $\tilde F_i$ to a subsequential $C^{2,\beta}_\textrm{loc}$ limit for any $\beta<\alpha$. To be precise, we claim that up to passing to a subsequence, for $\Omega\Subset S$ the maps $\tilde F_i \circ G_i : (\Omega,s) \to B(0,r_i^{-1}Q^{-1}) \subset \mathbf{R}^4$ converge in the $C^{2,\beta}$-topology. It is convenient prove this with respect to the flat metric on $\mathbf{R}^4$; in particular, this is just a statement about $\mathbf{R}^4$ valued functions on $(\Omega,s)$ as opposed to a statement about maps between Riemannian manifolds. 
	
	To prove this claim, it suffices to obtain $C^{2,\alpha}$ estimates of the functions $\tilde x_k : = x_k  \circ \tilde F_i \circ G_i$ for $k=1,2,3,4$ (with respect to the metrics $G^*_i \tilde h_i$). Indeed, this follows from the fact that $D^2_{G_i^* \tilde F_i^* \delta} \tilde x_k = (\partial_k \cdot \nu_\delta) A_{G_i \circ \tilde F_i,\delta} $ and that the (Euclidean) second fundamental form $A_{G_i \circ \tilde F_i,\delta} $ is locally uniformly in $C^{\alpha}$ (this follows in an straightforward way from the fact that the local graphical functions $f_i$ are locally uniformly bounded in $C^{2,\alpha}$).

	Thus, up to passing to a subsequence, the immersions $\tilde F_i$ limit to $F : (S,s) \to (\mathbf{R}^4,0)$ in the $C^{2,\beta}_\textrm{loc}$-sense as described above. This (and the $C^{1,\alpha}_\textrm{loc}$ convergence of $\tilde g_i$ to $\delta$) implies that $F$ is a complete stable minimal immersion with $|A_F(s)| = 1$. By \cite[Theorem 1]{CLstable}, such an immersion must be flat. This contradiction completes the proof. 
\end{proof}

Proposition \ref{prop:bd.geo} shows that $|\textnormal{sec}_g|\leq K$ implies $Q$-weakly bounded geometry for $Q=Q(K)$, so we have also proven the following:
\begin{corollary}
	Let $(X^{4},g)$ be a complete 4-manifold with $|\textnormal{sec}_g|\leq K$. Then there is a constant $C(K) < \infty$ such that every compact two-sided stable minimal immersion $M^{3}\to (X^{4},g)$ satisfies
	\[
	\sup_{q \in M} |A_{M}(q)| \min\{1,d_{M}(q,\partial M)\} \leq C(K).
	\]
\end{corollary}
This corollary will not be used in this paper, but we have included it since it resolves the conjecture of Schoen stated in \cite[Conjecture 2.13]{CM:book} (the analogous $3$-dimensional result was proven by Schoen in \cite{Schoen:estimates}, see also \cite{RST}). As mentioned above, the result in \cite[Theorem 3]{CLstable} established such an estimate only for closed $(X^{4},g)$ where $C=C(X,g)$ (see also \cite[\S 3]{White:notes}).

\section{Preliminary Results on Ends}\label{sec:ends.prelim}
We establish notation and collect some relevant facts about ends.  

\subsection{Ends adapted to geodesic balls}
Let $M$ be a complete Riemannian manifold. Fix a point $x \in M$, and take a length scale $L > 0$.

\begin{definition}
A collection of open sets $\{E_k\}_{k \in \N}$ is an \emph{end of $M$ adapted to $x$ with length scale $L$} if each set $E_k$ is an unbounded connected component of $M \setminus \overline{B}_{kL}(x)$ and satisfies $E_{k+1} \subset E_k$.
\end{definition}

\begin{proposition}\label{prop:connectedness}
If $M$ is simply connected and $\{E_k\}$ is an end adapted to $x$ with length scale $L$, then both $\overline{E}_k \setminus E_{k+1}$ and $\partial E_k$ are connected for all $k$.
\end{proposition}
\begin{proof}

Suppose $\partial E_k$ has at least two components. Since $B_{kL}(x)$ and $E_k$ are connected, we can construct a loop in $M$ having nontrivial intersection number with two of the components of $\partial E_k$, which contradicts that $M$ is simply connected. The connectedness of $\overline{E}_k\setminus E_{k+1}$ follows from the connectedness of $\partial E_k$ by taking segments of radial geodesics from $x$, which must intersect $\partial E_k$.

\end{proof}

\subsection{Parabolicity and nonparabolicity}
We recall the notion of parabolicity for subsets.

\begin{definition}\label{defi:parabolic}
Let $M$ be a complete Riemannian manifold. Let $K \subset M$ be a compact subset of $M$. Let $E \subset M$ be an unbounded component of $M \setminus K$ with smooth boundary. We say that $E$ is \emph{parabolic} if it does not admit a positive harmonic function $f$ satisfying
\[ f|_{\partial E}\ \equiv 1 \ \ \text{and}\ \ f|_E\ < 1. \]
Otherwise $E$ is \emph{nonparabolic}.
\end{definition}

We first state a well-known result about parabolic sets.

\begin{proposition}\label{prop:parabolicends}
Let $M$ be a complete Riemannian manifold. Let $K \subset M$ be a compact subset of $M$. Let $E \subset M$ be an unbounded component of $M \setminus K$ with smooth boundary. Suppose $E$ is parabolic. Let $f_i$ be harmonic functions on $E \cap B_{R_i}$ satisfying
\[ f_i|_{\partial E}\ \equiv 1\ \ \text{and}\ \ f_i|_{\partial B_{R_i}(x)}\ \equiv 0, \]
where $R_i \to \infty$ and $K \subset B_{R_i}(x)$ for all $i$. Then $f_i$ converges uniformly to 1 on compact subsets, and
\[ \lim_{i \to \infty} \int_{E} |\nabla f_i|^2 = 0. \]
\end{proposition}
\begin{proof}
We follow \cite[Theorem 10.1]{Li:harmonic.lectures}. The maximum principle guarantees that $0\leq f_{i}\leq 1$. Thus, up to a subsequence, $f_{i}$ limits to $f$ locally smoothly, with $\Delta f=0$, $0 \leq f \leq 1$, and $f|_{\partial E} = 1$. By parabolicity,  $f(x) = 1$ for some $x \in E$, so $f\equiv 1$ by the maximum principle. Thus, we see that $f_{i}$ limits locally smoothly to $1$ on $E$. Finally, integrating by parts, we see that
\[
\int_{E} |\nabla f_{i}|^{2} = - \int_{\partial E} f_{i} \nabla_{\nu} f_{i} \to 0.
\]
This completes the proof. 
\end{proof}

We now discuss when nonparabolicity is inherited by subsets.

\begin{proposition}\label{prop:nonparabolicends}
Let $M$ be a complete Riemannian manifold. Let $K \subset \tilde{K} \subset M$ be compact subsets of $M$ with smooth boundary. Let $E \subset M$ be an unbounded component of $M \setminus K$. If $E$ is nonparabolic, then there is a nonparabolic unbounded component $\tilde{E}$ of $E \setminus \tilde{K}$. 
\end{proposition}
\begin{proof}
Since $E$ is nonparabolic, there is a positive harmonic function $f$ on $E$ such that
\[  f|_{\partial E}\ \equiv 1 \ \ \text{and}\ \ f|_E\ < 1. \]
By addition and scaling, we can assume without loss of generality that
\[ \lim_{r \to \infty} \inf_{\partial B_r(x) \cap E} f = 0. \]
Since $\tilde{K}$ is compact with smooth boundary, $\partial \tilde{K}$ is a disjoint union of finitely many smooth closed surfaces. Since $E$ is connected, the boundary of each component of $E \setminus \tilde{K}$ is a union of at least one of the components of $\partial \tilde{K}$. Since distinct components of $E \setminus \tilde{K}$ have disjoint boundaries, the number of such components is finite.
Hence, there is an unbounded component $\tilde{E}$ of $E \setminus \tilde{K}$ such that
\[ \lim_{r \to \infty} \inf_{\partial B_r(x) \cap \tilde{E}} f = 0. \]
We show that $\tilde{E}$ is nonparabolic.

We take a sequence of harmonic functions $\tilde{f}_i$ on $\tilde{E} \cap B_{R_i}(x)$ satisfying
\[ \tilde{f}_i|_{\partial \tilde{E}} \ \equiv 1 \ \ \text{and}\ \ \tilde{f}_i|_{\partial B_{R_i}(x)} \equiv 0, \]
where $R_i \to \infty$ such that $\tilde{K} \subset B_{R_i}(x)$ for all $i$. Then $\tilde{f}_i$ converges locally uniformly to a positive harmonic function $\tilde{f}$ on $\tilde{E}$ with
\[ \tilde{f} |_{\partial \tilde{E}}\ \equiv 1 \ \ \text{and}\ \ \tilde{f}|_{\tilde{E}}\ \leq 1. \]
Let $c = \inf_{\partial \tilde{E}} f > 0$. Then $\tilde{f}_i \leq c^{-1}f$ on $\tilde{E}\cap B_{R_i}(x)$ for all $i$ by the maximum principle, which implies $\tilde{f} \leq c^{-1}f$. In particular,
\[ \lim_{r \to \infty} \inf_{\partial B_r(x) \cap \tilde{E}} \tilde{f} = 0, \]
so $\tilde{f} \not\equiv 1$. Hence, $\tilde{f}|_{\tilde{E}} < 1$ by the maximum principle, so $\tilde{E}$ is nonparabolic.
\end{proof}

By Proposition \ref{prop:nonparabolicends}, it makes sense to define nonparabolicity for an end.

\begin{definition}\label{def:par_ends}
An end $\{E_k\}$ adapted to $x$ with length scale $L$ is \emph{nonparabolic} if there exist connected open sets $F_k\subset E_k$ with smooth boundary such that
\[ E_k \setminus \overline{B}_{(k+1)L}(x) \subset F_k \subset E_k \]
and $F_k$ is nonparabolic for all $k$ sufficiently large.
\end{definition}

\section{Nonparabolic Ends of Stable Hypersurfaces in 4-Manifolds with $\Ric_{2}\geq 0$}\label{sec:non.par.stab}
In this section, we observe\footnote{One may observe that the same result holds (with the same proof) for $M^{n-1}\to (X^{n},g)$ when $(X^n,g)$ has $\Ric_{n-2}^g\geq 0$ (meaning that if $\Pi_{1},\dots,\Pi_{n-2}$ is a set of planes in $T_{p}X$ meeting pairwise orthogonally along a fixed line, then $\sum_{i=1}^{n-2}\text{sec}_g(\Pi_{i})\geq 0$).} that a minor modification of the arguments used in \cite{LiWang:ends.stable.minimal}, \cite[\S 11]{Li:harmonic.lectures} restricts the number of nonparabolic ends of a stable hypersurface in a 4-manifold with $\Ric_{2}\geq 0$.

\begin{theorem}\label{thm:onenonparabolic}
Let $(X^4, g)$ be a complete 4-manifold with $\Ric_2^g\geq 0$. Let $M^3 \to (X^4, g)$ be a complete two-sided stable minimal immersion. Let $K \subset M$ be a compact subset of $M$ with smooth boundary. Then $M \setminus K$ admits at most one nonparabolic unbounded component. In particular, $M$ has at most one nonparabolic end.
\end{theorem}

We give the proof in Section \ref{subsec:proof.thm.one.nonpar} after establishing some preliminary results. 
\subsection{Schoen-Yau inequality}
We use the following inequality of Schoen and Yau, which is known in non-negative sectional curvature. We show that the same proof works under the weaker assumption $\Ric_{2}\geq 0$. We include the proof for completeness.

\begin{lemma}[\cite{SY:harmonic.stable.minimal}]\label{lem:schoenyau}
Let $(X^4, g)$ be a complete 4-manifold satisfying $\Ric_{2}^g\geq 0$. Let $M^3 \to (X^4, g)$ be a complete two-sided stable minimal immersion. Let $u$ be a harmonic function on $M$. Then
\[ \frac{1}{3}\int_M \phi^2 |A_M|^2 |\nabla u|^2 + \frac{1}{2} \int_M \phi^2 |\nabla|\nabla u||^2 \leq \int_M |\nabla \phi|^2 |\nabla u|^2 \]
for any compactly supported, nonnegative $\phi \in W^{1,2}(M)$.
\end{lemma}
\begin{proof}
We choose a favorable test function in the stability inequality, combined with a rearrangement of the Bochner formula. Below, we assume that $u$ is not a constant function (otherwise the desired inequality is trivial). 

The favorable test function is $\psi := \phi |\nabla u|$, where $\phi$ is a compactly supported, nonnegative function in $W^{1,2}(M)$. Since $\Ric_2^g \geq 0$ implies non-negative Ricci curvature, the stability inequality gives
\begin{align}\label{eq:stability}
\int_M \phi^2 |A_M|^2 |\nabla u|^2
& \leq \int_M |\nabla \phi|^2|\nabla u|^2 + 2\int_M \phi |\nabla u| \langle \nabla \phi, \nabla |\nabla u|\rangle + \int_M \phi^2 |\nabla |\nabla u||^2\\
& =  \int_M |\nabla \phi|^2|\nabla u|^2 + \frac 1 2\int_M \langle \nabla \phi^2, \nabla |\nabla u|^2 \rangle + \int_M \phi^2 |\nabla |\nabla u||^2\\
& = \int_M |\nabla \phi|^2|\nabla u|^2 - \int_M \phi^2\left(  \frac 12  \Delta |\nabla u|^2 - |\nabla |\nabla u||^2 \right) . \notag
\end{align}
where the divergence theorem is used in the third equality.

We now rearrange the Bochner formula. Recall that the usual Bochner formula gives
\begin{equation} \label{eq:bochner}
\frac 12 \Delta |\nabla u|^2 = \mathrm{Ric}_M(\nabla u, \nabla u) + |\mathrm{Hess}_Mu|^2.
\end{equation}

%We assume that $u$ is not a constant function (otherwise the desired inequality is trivial). It suffices to consider points where $\nabla u \neq 0$; henceforth, we work at such a point. 

For now, consider a point $p$ where $\nabla u \neq 0$. Choose a local orthonormal frame around $p$ so that $e_1 = \frac{\nabla u}{|\nabla u|}$. The improved Kato inequality (cf.\ \cite[Lemma 7.2]{Li:harmonic.lectures}) yields
\begin{equation}\label{eq:hessterm}
|\mathrm{Hess}_Mu|^2 \geq \frac{3}{2} |\nabla|\nabla u||^2.
\end{equation}
Now we rearrange the Ricci curvature term. The Gauss equation and $\Ric_{2}\geq 0$ give
\begin{equation}\label{eq:gauss}
\mathrm{Ric}_M(e_1, e_1) = \sum_{j=2}^{3} R(e_j, e_1, e_1, e_j) - \sum_{j=1}^{3} A(e_1, e_j)^2
\geq - \sum_{j=1}^3 A(e_1, e_j)^2.
\end{equation}
On the other hand, Cauchy-Schwarz and minimality give
\begin{align}\label{eq:normA}
|A_M|^2
& \geq A_M(e_1, e_1)^2 + 2\sum_{j = 2}^{3} A_M(e_1, e_j)^2 + \sum_{j =2}^{3} A_M(e_j, e_j)^2\\
& \geq A_M(e_1, e_1)^2 + 2\sum_{j =2}^{3} A_M(e_1, e_j)^2 + \frac{1}{2}\left(\sum_{j=2}^{3} A_M(e_j, e_j)\right)^2\notag\\
& = A_M(e_1, e_1)^2 + 2\sum_{j =2}^{3} A_M(e_1, e_j)^2 + \frac{1}{2}A_M(e_1, e_1)^2\notag\\
& \geq \frac{3}{2}\sum_{j=1}^{3} A_M(e_1, e_j)^2.\notag
\end{align}
Together, (\ref{eq:gauss}) and (\ref{eq:normA}) give
\begin{equation}\label{eq:ricterm}
\mathrm{Ric}_{M}(\nabla u, \nabla u) \geq -\frac{2}{3}|A_M|^2|\nabla u|^2.
\end{equation}
Hence, (\ref{eq:bochner}), (\ref{eq:hessterm}), (\ref{eq:ricterm}) give the rearrangement
\begin{equation}\label{eq:rearrangebochner}
\frac{1}{2}\Delta |\nabla u|^2 \geq -\frac{2}{3}|A_M|^2|\nabla u|^2 + \frac{3}{2}|\nabla |\nabla u||^2,
\end{equation}
so
\begin{equation}\label{eq:final-boch}
\frac 12 \Delta |\nabla u|^2 -  |\nabla |\nabla u||^2 \geq -\frac{2}{3}|A_M|^2|\nabla u|^2 + \frac{1}{2} |\nabla |\nabla u||^2. 
\end{equation}
We derived \eqref{eq:final-boch} under the assumption $\nabla u \neq 0$ at the given point. 

Now suppose that $|\nabla u|$ vanishes at $p$ but $x\mapsto |\nabla u|(x)$ is differentiable at $p$. In this case, the right hand side of \eqref{eq:final-boch} is $=0$ (since $p$ is a local minimum of $|\nabla u|$). On the other hand, the left-hand side is $\geq 0$ since $p$ is a local minimum of the smooth function $|\nabla u|^2$. In sum, we find that that \eqref{eq:final-boch} holds anytime that $|\nabla u|$ is differentiable, which is almost everywhere (by Rademacher's theorem). Thus, we can use  \eqref{eq:final-boch} in \eqref{eq:stability} to complete the proof.
\end{proof}

\subsection{Multiple nonparabolic ends}
Under the assumption of multiple nonparabolic ends with respect to a fixed compact set, we produce a non-constant harmonic function.

\begin{lemma}[{cf.\ \cite[Theorem 4.3]{Li:harmonic.lectures}}]\label{lem:twononparabolic}
Let $M$ be a complete Riemannian manifold. Let $K \subset M$ be a compact subset of $M$ with smooth boundary. Suppose $M \setminus K$ has at least two nonparabolic unbounded components. Then there exists a non-constant harmonic function with finite Dirichlet energy on $M$.
\end{lemma}
\begin{proof}
Let $E,F$ denote nonparabolic components of $M\setminus K$. There exists a harmonic function $0 \leq h \leq 1$ on $M\setminus K$ so that $h\equiv 0$ on $M\setminus (K \cup E \cup F)$, $h|_{\partial E} =1$, $h|_{\partial F} = 0$, $\liminf_{x\in E,x\to\infty} h = 0$, and $\limsup_{x\in F,x\to\infty}h = 1$. By the argument in \cite[Lemma 3.6]{Li:harmonic.lectures}, we can assume that $h$ has finite Dirichlet energy on $M\setminus K$.

For $R_{i}\to \infty$, define $f_{i}$ to be the harmonic function on $B_{R_{i}}(x)$ with $f_{i}|_{\partial B_{R_{i}}(x)} = h|_{\partial B_{R_{i}}(x)}$. Note that $0\leq f_{i}\leq 1$. In particular, $f_{i}\leq h$ on $E\cap B_{R_{i}}(x)$, and $f_{i}\geq h$ on $F\cap B_{R_{i}}(x)$. Passing to a subsequence, the $f_{i}$ limit locally smoothly to a harmonic function $f$ on $M$ with $f \leq h$ on $E$ and $f\geq h$ on $F$. In particular, $f$ is non-constant. 

Finally, observe that
\[
\int_{B_{R_{i+1}}(x)} |\nabla f_{i+1}|^{2} \leq \int_{B_{R_{i}}(x)} |\nabla f_{i}|^{2} + \int_{B_{R_{i+1}}(x)\setminus B_{R_{i}}(x)} |\nabla h|^{2}.
\]
This implies that $f$ has finite Dirichlet energy, completing the proof.
\end{proof}

\subsection{Proof of Theorem \ref{thm:onenonparabolic}}\label{subsec:proof.thm.one.nonpar}
\begin{proof}
Suppose for contradiction that $\{E_k\}$ and $\{\tilde{E}_k\}$ are distinct nonparabolic ends of $M$ adapted to $x$ with length scale $L$. Let $F_k$ and $\tilde{F}_k$ be open subsets with smooth boundary given by Definition \ref{def:par_ends}. Since the ends are distinct, there is a $k_0 \in \N$ such that $E_{k_0} \cap \tilde{E}_{k_0} = \varnothing$. Taking $k_0$ larger if necessary, $F_{k_0}$ and $\tilde{F}_{k_0}$ are disjoint and nonparabolic. Hence, it suffices to rule out two nonparabolic sets.

Suppose $E$ and $F$ are nonparabolic unbounded components of $M\setminus K$. By Lemma \ref{lem:twononparabolic}, there is a nonconstant harmonic function $u$ with finite Dirichlet energy on $M$.

We choose a nice cutoff function to use in Lemma \ref{lem:schoenyau}. Let $\rho_i$ be a smoothing of $d_M(x,\cdot)$ with $\rho_i\mid_{\partial B_{(k_0+i)L}(x)} = (k_0+i)L$, $\rho_i\mid_{\partial B_{(k_0+2i)L}(x)} = (k_0+2i)L$, and $|\nabla \rho_i| \leq 2$. Then we define
\[ \phi_i(y) := \begin{cases}
1 & y \in \overline{B}_{(k_0+i)L}(x)\\
\frac{(k_0+2i)L - \rho_i(y)}{iL} & y \in \overline{B}_{(k_0+2i)L}(x) \setminus B_{(k_0+i)L}(x)\\
0 & y \in M \setminus B_{(k_0+2i)L}(x).
\end{cases} \]
Lemma \ref{lem:schoenyau} and the fact that $u$ has finite Dirichlet energy implies
\begin{align*}
\int_{B_{(k_0+i)L}(x)} \frac{1}{3}|A_M|^2|\nabla u|^2 + \frac{1}{2}|\nabla|\nabla u||^2
& \leq \int_M \frac{1}{3}\phi_i^2|A_M|^2|\nabla u|^2 + \frac{1}{2}\phi_i^2|\nabla|\nabla u||^2\\
& \leq \int_M |\nabla \phi_i|^2|\nabla u|^2\\
& \leq \frac{C}{i^2}.
\end{align*}
Taking $i \to \infty$, we conclude that
\[ |A_M|^2|\nabla u|^2 \equiv 0 \ \ \text{and}\ \ |\nabla |\nabla u||^2 \equiv 0. \]
Since $|\nabla u|$ is constant and $u$ is non-constant, we have $|\nabla u| \neq 0$ everywhere. Then $|A_M| \equiv 0$, which implies  $\mathrm{Ric}_M \geq 0$ by Lemma \ref{lem:k2+}. The assumption of at least two ends implies (by Cheeger-Gromoll splitting) that $M$ is a product $\mathbf{R} \times P$. Therefore, $M$ has infinite volume. However, $|\nabla u|^2$ is a nonzero constant and $\int_M |\nabla u|^2 < \infty$ by Lemma \ref{lem:twononparabolic}, which implies that $M$ has finite volume. Hence, we reach a contradiction.
\end{proof}

\section{Stable Hypersurfaces in PSC 4-Manifolds with $\Ric_{2}\geq 0$}\label{sec:proofs.main}

Our aim in this section is to prove Theorem \ref{thm:main} (which immediately implies Theorem \ref{thm:main.sec} and Corollary \ref{cor:cpt.sec}).

\begin{remark}
If $M$ is parabolic, then the conclusion of Theorem \ref{thm:main} follows immediately, as parabolicity allows us to ``plug 1 into the stability inequality.'' The hard case is therefore the nonparabolic case. In this case, Theorem \ref{thm:onenonparabolic} implies that $M$ has precisely one nonparabolic end. Hence, the main difficulty of Theorem \ref{thm:main} is this troublesome end.
\end{remark}

\begin{remark}
We can assume without loss of generality that $M$ is simply connected. In the general case, we pass to the universal cover, which inherits minimality, stability, two-sidedness, and completeness from the original immersion. The conclusion of Theorem \ref{thm:main} for the universal cover then descends to the original immersion.
\end{remark}

\subsection{$\mu$-bubbles}

We first recall a diameter bound that uses the theory of warped $\mu$-bubbles for stable minimal hypersurfaces in PSC 4-manifolds.

\begin{lemma}[Warped $\mu$-bubble diameter bound]\label{lem:mububble}
Let $(X^4, g)$ be a complete 4-manifold with $R_g \geq 1$. Let $N^3 \to (X^4, g)$ be a two-sided stable minimal immersion with compact boundary. There are universal constants $L > 0$ and $c > 0$ such that if there is a $p \in N$ with $d_N(p, \partial N) > L/2$, then in $N$ there is an open set $\Omega \subset B_{L/2}^N(\partial N)$ and a smooth surface $\Sigma^2$ such that $\partial \Omega = \Sigma \sqcup \partial N$ and each component of $\Sigma$ has diameter at most $c$.
\end{lemma}
\begin{proof}
This follows from Gromov's  band-width estimate technique \cite{Gromov:metric-inequalities}. For completeness, we sketch the proof here, with references to the relevant statements from \cite{CL:aspherical}. Choose $L=20\pi$. Since $N$ is two-sided and stable, we have
\[\int_N |\nabla \psi|^2 - \tfrac12 (R_g - R_N + |A_N|^2) \psi^2 \ge 0,\quad \forall \psi\in C_c^1 (N).\]
Since $R_g \ge 1$, there exists $u\in C^\infty (N)$, $u>0$ in $\mathring{N}$ such that
\begin{equation}\label{eq:first.eigenfunction.stability}
	\Delta_N u \le -\tfrac 12 (1-R_N) u.
\end{equation}
For instance, see \cite[Theorem 1]{fischer-colbrie-schoen}.

Take $\rho_0\in C^\infty(M)$ to be a smoothing of $d_N(\cdot, \partial N)$, such that $|\Lip(\rho_0)|\le 2$, and $\rho_0=0$ on $\partial N$. Choose $\eps\in (0,\tfrac12)$ such that $\eps,\ 4\pi + \frac{3}{2}\eps,\ 8\pi + 2\eps$ are regular values of $\rho_0$. Define
\[\rho = \frac{\rho_0 - \eps}{8 + \frac{\eps}{\pi}} - \frac{\pi}{2},\]
$\Omega_1 = \{x\in N: -\tfrac{\pi}{2}< \rho < \tfrac{\pi}{2}\}$, and $\Omega_0 = \{x\in N: -\tfrac{\pi}{2}< \rho \le 0\}$. Clearly $|\Lip(\rho)|<\tfrac 14$.

On $\Omega_1$, define $h(x)=-\tfrac12 \tan(\rho(x))$. For Caccioppoli sets $\Omega$ in $\Omega_1$ such that $\Omega\Delta \Omega_0$ is compactly contained in $\Omega_1$, consider
\[\cA(\Omega) = \int_{\partial \Omega} u ~d\cH^2 - \int_{\Omega_1}(\chi_\Omega - \chi_{\Omega_0}) hu ~d\cH^3.\]
By \cite[Proposition 12]{CL:aspherical}, there exists a minimizer $\tilde \Omega$ for $\cA$ such that $\tilde \Omega\Delta \Omega_0$ is compactly contained in $\Omega_1$. Let $\Omega= \{x\in \N: 0\le \rho_0(x)\le \eps\}\cup \tilde \Omega$. We claim that $\Omega$ satisfies the conclusion with $c=\frac{4\pi}{\sqrt 3}$. 

Indeed, let $\Sigma_0$ be a connected component of $\Sigma = \partial \Omega\cap \Omega_1$. The nonnegativity of the second variation for $\cA$ on $\Sigma_0$ implies that (see \cite[Lemma 14]{CL:aspherical})	
\begin{multline}\label{eq:mububble.stability}
	\int_{\Sigma_0} |\nabla_{\Sigma_0}\psi|^2 u -\tfrac12 \left(R_N -\tfrac14 -2K_{\Sigma_0}\right)\psi^2  u +(\Delta_N u - \Delta_{\Sigma_0} u)\psi^2 \\
		-\tfrac12 u^{-1} \langle \nabla_N u,\nu\rangle^2 \psi^2 - \tfrac12 \left(\tfrac14 + h^2 + 2\langle \nabla_N h ,\nu\rangle\right)\psi^2 u \ge 0,\quad \forall \psi\in C^1(\Sigma_0).
\end{multline}
The choice of $\rho$ and $h$ guarantees that $\tfrac 14 + h^2 + 2\langle \nabla_N h,\nu\rangle \ge 0$ pointwise on $\Omega_1$. Combined with \eqref{eq:first.eigenfunction.stability}, this implies that
\[\int_{\Sigma_0} |\nabla_{\Sigma_0} \psi|^2 u - \tfrac 12 \left(\tfrac 34  - 2K_{\Sigma_0}\right)\psi^2 u - (\Delta_{\Sigma_0}u) \psi^2 \ge 0,\quad \forall \psi\in C^1(\Sigma_0).\]
Thus we apply \cite[Lemma 16 and Lemma 18]{CL:aspherical} and conclude that $\diam (\Sigma_0)\le \frac{4\pi}{\sqrt 3}$. This completes the proof.
\end{proof}

\subsection{Almost linear volume growth}

We show that every end of a stable hypersurface in a $\Ric_{2}\geq 0$ PSC 4-manifold has a core tube with linear volume growth.

\begin{lemma}[Almost linear volume growth of an end]\label{lem:linvolgrowth}
Let $(X^4, g)$ be a complete 4-manifold with weakly bounded geometry, $\Ric_{2}^g\geq 0$, and $R_g \geq 1$. Let $M^3 \to (X^4, g)$ be a simply connected complete two-sided stable minimal immersion. Let $\{E_k\}_{k \in \N}$ be an end of $M$ adapted to $x \in M$ with length scale $L$, where $L$ is the constant from Lemma \ref{lem:mububble}. Let $M_k := E_k \cap B_{(k+1)L}(x)$. Then there is a constant $C > 0$ such that
\[ \mathrm{Vol}_M(M_k) \leq C \]
for all $k$.
\end{lemma}
\begin{proof}
%Take $k$ large enough so that Proposition \ref{prop:connectedness} holds for $k-1$.

We apply Lemma \ref{lem:mububble} to $E_k$, which supplies an open set $\Omega_k$ in $B_{L/2}(\partial E_k) \cap E_k$ with $\partial E_k \subset \partial \Omega_k$. Note that $\Omega_k \subset M_k$ and $\partial E_{k+1} \cap \overline{\Omega_k} = \varnothing$ by construction. Let $\Sigma_k^{(i)}$ denote the components of $\partial \Omega_k \setminus \partial E_k$.

First, we claim that $\overline{M}_k$ is connected. Indeed, this claim follows from the connectedness of $\partial E_k$ by taking segments of radial geodesics from $x$, just as in Proposition \ref{prop:connectedness}.

%Indeed, we show that any point $y \in \overline{M}_k$ can be joined by a curve $\gamma$ in $\overline{M}_k$ to $\partial E_k$, which is connected by Proposition \ref{prop:connectedness}. Let $\tilde{\gamma}$ be a minimizing radial geodesic from $x$ to $y$. An arc $\gamma$ of $\tilde{\gamma}$ joins $y$ to $\partial B_{kL}(x)$ because $d(x,y) \geq kL$. Since $\overline{E}_k$ is the component of $M \setminus B_{kL}(x)$ containing $y$, $\gamma$ joins $y$ to $\partial E_k$ inside $\overline{E}_k$. Moreover, $\gamma$ stays in $\overline{B}_{(k+1)L}(x)$ because $d(x,y) \leq (k+1)L$. Hence, $\gamma \subset \overline{M}_k = \overline{E}_k \cap \overline{B}_{(k+1)L}(x)$.

Second, we claim that some $\Sigma_k^{(i)}$ separates $\partial E_k$ from $\partial E_{k+1}$. Let $\gamma$ be any curve from $\partial E_k$ to $\partial E_{k+1}$ in $\overline{M}_k$, which exists by the connectedness of $\overline{M}_k$. Perturbing $\gamma$ if necessary, we can guarantee that any intersection with $\bigcup_i \Sigma_k^{(i)}$ is transverse. Since $\partial E_k \subset \overline{\Omega}_k$ and $\partial E_{k+1} \cap \overline{\Omega}_k = \varnothing$, $\gamma$ intersects some $\Sigma_k^{(i)}$ an odd number of times. Suppose for contradiction that $\Sigma_k^{(i)}$ does not separate $\partial E_k$ from $\partial E_{k+1}$. Then there is another curve $\gamma'$ from $\partial E_k$ to $\partial E_{k+1}$ in $\overline{M}_k \setminus \Sigma_k^{(i)}$. Since $\partial E_k$ and $\partial E_{k+1}$ are connected by Proposition \ref{prop:connectedness}, we have constructed a loop with nontrivial intersection number with $\Sigma_k^{(i)}$, contradicting the simple connectedness of $M$.

Let $\Sigma_k$ denote the component of $\partial \Omega_k \setminus \partial E_k$ provided by the above claim. By Lemma \ref{lem:mububble}, we have $\mathrm{diam}(\Sigma_k) \leq c$.

We claim that there is a constant $D > 0$ (independent of $k$) so that $\mathrm{diam}(M_k) \leq D$. Let $y$ and $z$ be any points in $M_k$. Take the radial geodesic from $x$ to $y$. As in the proof of the connectedness of $M_k$, this curve intersects $\partial E_k$ and $\partial E_{k-1}$. By the choice of the components $\{\Sigma_k\}$, an arc of this curve of length at most $2L$ joins $y$ to $\Sigma_{k-1}$. The same argument applies to $z$. Hence, the diameter bound for $\Sigma_{k-1}$ implies
\[ d_M(y, z) \leq 4L + c =: D. \]

By Lemma \ref{lem:boundedgeometry}, $M$ has bounded second fundamental form (e.g.\ use the compact immersion of $B^M_2(q)$ for any $q \in M$). By Lemma \ref{lem:k2+}, $\mathrm{Ric}_M$ is bounded from below. Hence, by Bishop--Gromov volume comparison, there is a constant $C > 0$ so that
\[ \mathrm{Vol}(B^M_{D}(p)) \leq C \]
for all $p \in M$. Since $M_k \subset B^M_{D}(p)$ for any $p \in M_k$, the lemma follows.
\end{proof}

\subsection{Decomposition of $M$}\label{subsec:decomp}
To construct a nice sequence of test functions for the stability inequality, we need to decompose $M$ into convenient building-blocks.

Let $M^3$ be a complete, simply connected, Riemannian 3-manifold such that the conclusions of Theorem \ref{thm:onenonparabolic} hold. Let $\{E_k\}_{k \in \N}$ be the nonparabolic end of $M$ adapted to $x \in M$ with length scale $L$ (where $L$ is the universal constant from Lemma \ref{lem:mububble}).

We now decompose $M$ to handle the nonparabolic end. Let $k_0 \geq 1$ and $i \geq 1$.
\begin{itemize}
\item $M_k := E_k \cap B_{(k+1)L}(x)$ for all $k$.
\item $\{P_{k_0}^{(\alpha)}\}_{\alpha = 1}^{n_{k_0}}$ are the components of $M \setminus \overline{B}_{k_0L}(x)$ besides $E_{k_0}$.
\item $\{P_k^{(\alpha)}\}_{\alpha = 1}^{n_k}$ are the components of $E_{k-1} \setminus \overline{B}_{kL}(x)$ besides $E_k$ for $k > k_0$.
\end{itemize}
Note that we have the decomposition
\[ M = \overline{B}_{(k_0+i)L}(x) \cup \bigcup_{k = k_0+i}^{k_0 + 2i - 1} \left(\overline{M}_k \cup \bigcup_{\alpha=1}^{n_k} \overline{P}_k^{(\alpha)} \right) \cup (\overline{E}_{k_0 + 2i - 1} \setminus B_{(k_0+2i)L}(x)).  \]
By Theorem \ref{thm:onenonparabolic}, $P_k^{(\alpha)}$ is either bounded or $P_k^{(\alpha)} \setminus K$ is parabolic for all compact $K \subset M$ such that $P_k^{(\alpha)} \setminus K$ has smooth boundary.

\subsection{Proof of Theorem \ref{thm:main}}
\begin{proof}
Supposing it exists, let $\{E_k\}$ be the nonparabolic end of $M$. We use the decomposition of $M$ from \S \ref{subsec:decomp}. 

\emph{The Core Tube}. For each $k$, let $\rho_k$ be a smooth function on $\overline{M}_k$ such that $|\nabla \rho_k| \leq 2$,
\[ \rho_k |_{\partial E_k} \equiv kL,\ \ \text{and}\ \ \rho_k |_{\partial M_k \setminus \partial E_k} \equiv (k+1)L. \]
For instance, we can take $\rho_k$ to be a smoothing of the distance function from $x$.

Let $\phi : [(k_0+i)L, (k_0+2i)L] \to [0, 1]$ be the linear function satisfying $\phi((k_0+i)L)=1$ and $\phi((k_0+2i)L) = 0$.

\emph{The Extraneous Pieces}. By Proposition \ref{prop:parabolicends}, there is a compactly supported Lipschitz function $u_{k, \alpha,i}$ on $\overline{P}_k^{(\alpha)}$ such that
\[ u_{k, \alpha,i}|_{\partial P_k^{(\alpha)}}\ \equiv 1\ \ \text{and} \ \ \int_{P_k^{(\alpha)}} |\nabla u_{k, \alpha,i}|^2 < \frac{1}{i^2n_k}. \]
Note that if $\partial P_k^{(\alpha)}$ is not smooth, we apply Proposition \ref{prop:parabolicends} to $P_k^{(\alpha)} \setminus K$ (where $K$ is compact and $P_k^{(\alpha)} \setminus K$ has smooth boundary) and extend the function by 1 on $P_k^{(\alpha)} \cap K$. Moreover, if $P_k^{(\alpha)}$ is bounded, we can just take $u_{k, \alpha,i} \equiv 1$.

\emph{The Test Function}. We now construct a  test function $f_i$ for the stability inequality as follows:
\[ f_i(y) := \begin{cases}
1 & y \in \overline{B}_{(k_0+i)L}(x)\\
0 & y \in \overline{E}_{k_0+2i-1} \setminus B_{(k_0+2i)L}(x)\\
\phi(\rho_k(y)) & y \in \overline{M}_k,\ k_0+i \leq k < k_0 + 2i\\
\phi(kL)u_{k, \alpha,i} & y \in \overline{P}_k^{(\alpha)},\ k_0+i \leq k < k_0 + 2i.
\end{cases} \]

By construction, $f_i$ is compactly supported and Lipschitz. Therefore, $f_i$ is an eligible test function for the stability inequality. Hence, the stability inequality and Lemma \ref{lem:linvolgrowth} give
\begin{align*}
\int_M (\mathrm{Ric}_g(\nu, \nu) + & |A_M|^2)f_i^2
\leq \int_M |\nabla f_i|^2\\
& = \sum_{k=k_0+i}^{k_0+2i-1} \int_{M_k} \phi'(\rho_k)^2|\nabla \rho_k|^2 + \sum_{k=k_0+i}^{k_0+2i-1}\sum_{\alpha = 1}^{n_k} \phi(kL)^2\int_{P_k^{(\alpha)}} |\nabla u_{k,\alpha,i}|^2\\
& \leq \frac{4C}{iL^2} + \frac{1}{i} = \frac{C'}{i}.
\end{align*}
Since $f_i$ converges uniformly to the constant 1 function on compact subsets, the $i \to \infty$ limit yields the desired conclusion.
\end{proof}

\section{Topology of PSC $4$-manifolds with $\Ric_2 > 0$}\label{sec:topo}

In this section we prove Theorem \ref{thm:topo}. We begin with the following lemma.

\begin{lemma}\label{lem:topo}
	Suppose $X^{n}$ is a non-compact oriented manifold, $\sigma\subset X$ is a closed curve and $[\sigma]\in H_1(X)$ is a nontrivial, non-torsion element. For any pre-compact open set $D\subset X$ containing a neighborhood of $\sigma$, there exists a smooth compact submanifold $M_0$ of dimension $(n-1)$, such that $\partial M_0\subset X\setminus D$, and the algebraic intersection number of $M_0$ and $\sigma$ is $1$. 
\end{lemma}

\begin{proof}
By Poincar\'e duality for non-compact manifolds, $H_1(X)$ is isomorphic to $H_c^{n-1}(X)$. Thus, there exists a connected pre-compact open set $A$ with $D\subset A$ as well as $\alpha\in H^{n-1}(X,X\setminus A)$ with $\alpha \frown \mu_A = [\sigma]$. By assumption, $\alpha$ is not a torsion element of $H^{n-1}(X,X\setminus A)$. Thus, by the universal coefficient theorem, we can find $\beta\in H_{n-1}(X,X\setminus A)$ with $\alpha\frown \beta = 1\in H_0(X,X\setminus A)$. By the excision theorem and Lefschetz duality, we have
	\[H_{n-1}(X,X\setminus A)\simeq H^1(A)\simeq \langle A,S^1\rangle,\]
	where $\langle A, S^1\rangle$ is the basepoint-preserving homotopy classes of maps to $S^1$. Take a smooth map $f$ in $\langle A, S^1\rangle$ corresponding to $\beta$, and let $p\in S^1$ be a regular value of $f$. Then the smooth hypersurface $f^{-1}(p)\subset A$ represents $\beta$ and hence has algebraic intersection $1$ with $\sigma$.
\end{proof}

\begin{proof}[Proof of Theorem \ref{thm:topo}]
	We first prove that $H_1(X)$ consists only of torsion elements. If not,  there exists a closed curve $\sigma$ with $[\sigma]\in H_1(X)$ non-torsion. Take $p\in X$, $R_0>0$ with $\sigma\subset B_{R_0}(p)$. By Lemma \ref{lem:topo}, for any  integer $k>R_0$, there exists a smooth compact submanifold $\tilde M_k$ such that $\partial \tilde M_k\subset X\setminus B_{k}(p)$, and the algebraic intersection number of $\tilde M_k$ with $\sigma$ is $1$. Choose $\Omega_{k}\subset X$ a pre-compact region with smooth boundary, with $\tilde M_{k}\subset \Omega_{k}$. Form a metric $g_{k}$ so that $\Omega_{k}$ has mean-convex boundary and $g_{k}$ agrees with $g$ away from $\partial\Omega_{k}$. 
	
	Consider the area-minimizing problem
	\[\inf\{\cH^3_{g_{k}}(M): \partial M=\partial \tilde M_k, M \textrm{ is homologous to }\tilde M_k\}. \]
	It is standard to show that a compact smooth two-sided embedded minimizer $M_k$ exists. Moreover, since $M_k$ is homologous to $\tilde M_k$, the algebraic intersection number of $M_k$ with $\sigma$ is $1$. In particular, $M_k\cap \sigma\ne \emptyset$. 
	
	By Lemma \ref{lem:boundedgeometry}, we have curvature estimates for $\{M_k\}$ on any compact subset of $X$. Thus, by passing to a subsequence (not relabeled), $\{M_k\}$ converges to a complete two-sided stable minimal immersion $M^{3} \to (X,g)$. Note that $M$ is not empty, thanks to the condition that $M_{k}\cap\sigma \not = \emptyset$. This contradicts Theorem \ref{thm:main}.
	
	Now we prove that $X$ is $H_2$-trivial at infinity. Suppose the contrary, that there exists $\alpha\in H_2^\infty(X)$ whose image in $H_2(X)$ is zero. By definition, there exist nested bounded open sets $\{A_j\}_{j=1}^\infty$ such that $A_i\subset A_j$ when $i\le j$ and $\cup A_j = X$, and non-trivial elements $\alpha_j\in H_2(X\setminus A_j)$, such that:
	\begin{enumerate}
		\item When $i\le j$, $\iota_*(\alpha_j) = \alpha_i\ne 0$, here $\iota: X\setminus A_j \hookrightarrow X\setminus A_i$ is the inclusion.
		\item For each $j$, $\iota_* (\alpha_j ) = 0$, here $\iota: X\setminus A_j\to X$ is the inclusion.
	\end{enumerate}
	Thus, for each $j$, there exists a $2$-cycle $\Sigma_j\subset X\setminus A_j$ such that $[\Sigma_j]$ is null homologous in $X$ and if $\tilde M$ is a 3-chain with $\partial \tilde M=\Sigma_j$, then $M\cap A_1 \ne \emptyset$ (otherwise $[\Sigma_j]$ would be null-homologous in $X\setminus A_1$, contradiction). Consider the area-minimizing problem (for metrics $g_{j}$ as before)
	\[\inf\{\cH^3_{g_{j}}(\tilde M):\partial \tilde M = \Sigma_j \}.  \]
	A smooth compact two-sided minimizer $M_j$ exists. By construction, $M_j\cap A_1\ne \emptyset$. Thus, by Lemma \ref{lem:boundedgeometry}, a subsequence converges smoothly to a complete two-sided stable minimal immersion $M^3 \to (X^4, g)$, contradicting Theorem \ref{thm:main}. This completes the proof. 
\end{proof}

\appendix

\section{Curvature conditions} \label{sec:curv-cond}
Here we review the various curvature conditions referred to in this paper. Fix $(X^{n+1},g)$ a Riemannian manifold. We recall the convention used here: $\Pi \subset T_pX$ a $2$-plane with orthonormal basis $\{u,v\}$ then $\sec_g(\Pi) = R_g(v,u,u,v)$. 

For $k \in \{1,\dots,n\}$ and $v_0,\dots,v_k$ set of $k+1$ orthonormal vectors, we define the \emph{$k$-th intermediate Ricci curvature} by
\[
\Ric_k(v_0,\dots,v_k) : = \sum_{i=1}^k R(v_0,v_k,v_k,v_0). 
\]
Note that $\Ric_{n}$ is the usual Ricci curvature while $\Ric_1$ is the sectional curvature. 

Following \cite{ShenYe} we define the \emph{bi-Ricci curvature} of an orthonormal set of vectors $\{u,v\}$ by
\[
\biRic(u,v) = \Ric(u,u) + \Ric(v,v) - R(u,v,v,u). 
\]
Note that when $n+1=3$ then $\biRic(u,v) = R_g/2$, but in higher dimensions $\biRic\geq 0$ is a stronger curvature condition than $R_g\geq 0$.

\section{Examples} \label{sec:exam}

\subsection{Positive sectional curvature} \label{sec:pos.K.cx} Here we give the details of Example \ref{exam:hardy} concerning the existence of a complete stable minimal hypersurface in ambient positive sectional curvature. 

For $\alpha \in (0,1)$ to be chosen close to $1$ below, define 
\[
\rho(r) = \alpha r + (1-\alpha) \int_{0}^{r} e^{-s^{2}} ds.
\]
Consider a metric $g$ on $\mathbf{R}^{4}$ given by
\[
g = dr^{2} + \rho(r)^{2} g_{\mathbf{S}^{3}}.
\]
It is standard to compute (cf.\ \cite[\S 4.2.3]{Petersen:Riemannian}) that the sectional curvatures of $g_{\alpha}$ lie between 
\[
-\frac{\rho''(r)}{\rho(r)} = \frac{2(1-\alpha) r e^{-r^{2}}}{\alpha r + (1-\alpha) \int_{0}^{r} e^{-s^{2}} ds} 
\]
and 
\[
\frac{1-\rho'(r)^{2}}{\rho(r)^{2}} = \frac{1 - (\alpha + (1-\alpha)e^{-r^{2}})^{2}}{\left(\alpha r + (1-\alpha) \int_{0}^{r} e^{-s^{2}} ds\right)^{2}},
\]
so $(\mathbf{R}^{4},g)$ has positive sectional curvature. On the other hand, we can fix a totally geodesic $\mathbf{R}^{3} \to (\mathbf{R}^{4}, g)$ (corresponding to $[0,\infty)\times \mathbb{S}^{2} \subset [0,\infty)\times \mathbb{S}^{3}$ in the radial coordinates used above). We claim that this is a stable immersion (embedding), at least for $\alpha$ sufficiently close to $1$. Indeed, for $\varphi \in C^{\infty}_{c}(\mathbf{R}^{3})$, we compute (un-barred quantities denote the induced metric $dr^{2} + \rho(r)^{2} g_{\mathbf{S}^{2}}$ and barred quantities denote the Euclidean metric $dr^{2} + r^{2} g_{\mathbf{S}^{2}}$)
\begin{align*}
\int_{\mathbf{R}^{3}} |\nabla \varphi|^{2} d\mu & \geq  \int_{\mathbf{R}^{3}} \alpha^{2} |\bar\nabla \varphi|^{2}d\bar\mu \\
& \geq  \int_{\mathbf{R}^{3}} \frac{\alpha^{2}}{4r^{2}} \varphi^{2} d\bar\mu \\
& \geq  \int_{\mathbf{R}^{3}} \frac{\alpha^{2}}{4r^{2}} \varphi^{2} d\mu . 
\end{align*}
In the first and third lines we used $\alpha r \leq \rho(r) \leq r $, in the second we used a well-known\footnote{For any $1<p<n$ and $\varphi \in C^\infty_c(\mathbf{R}^n)$ it holds that $\int_{\mathbf{R}^n} \frac{u^p}{|x|^p} d\bar\mu \leq \left( \frac{p}{n-p} \right)^p \int_{\mathbf{R}^n} |\bar\nabla \varphi|^p d\bar\mu$. This can be deduced from the usual one-variable Hardy inequality \cite[Theorem 327]{HLP} via symmetrization (cf.\ \cite[Section 2]{CianchiFerone}). Alternatively, the $p=2$ case (all we need here) can be deduced via a simple integration by parts argument (cf.\ \cite[Section 5.8.4]{Evans}). } Hardy inequality. Finally, we note that 
\[
\Ric_{g}(\nu,\nu) \leq 2 \left(\frac{1-\alpha^{2}}{\alpha^{2}r^{2}}+\frac{1-\alpha}{\alpha} e^{-r^{2}} \right) \leq \frac{\alpha^{2}}{4r^{2}},
\]
where the second inequality holds for $\alpha$ sufficiently close to $1$. Hence, for this choice of $\alpha$, we find that $\mathbf{R}^{3}\to(\mathbf{R}^{4},g)$ is stable.

\subsection{Positive Ricci curvature and strictly positive scalar curvature} \label{sec:pos.norm.Ric}

The following example should be compared to the example of Schoen indicated in \cite[Appendix]{MeeksRosenberg}. We construct a complete two-sided stable minimal hypersurface $M^3 \to (X^{4},g)$ so that along $M$, $R_{g}\geq 1,\Ric_{g} \geq 1$, and the sectional curvature is uniformly bounded below.

\begin{remark}
We emphasize that $(X^4,g)$ will not be complete, but $M^3$ will have a tubular neighborhood of uniform diameter. We discuss this point further below. 
\end{remark}

\begin{remark}
The example below can be directly generalized to produce $M^{n-1}\to (X^{n},g)$ as above, for any dimension $n\geq 4$. 
\end{remark}

Fix $M_{0} = (S^{1}\times S^{2}) \# (S^{1}\times S^{2})$. Because $\dim M_{0} = 3 > 2$, the Schoen--Yau/Gromov--Lawson surgery construction \cite{SY:descent,GromovLawson} and the Kazdan--Warner trichotomy theorem \cite{KW:conf,KW:exist,KW:direct} (see e.g., \cite[Theorem 1.3]{KW:conf}) imply that there is a scalar flat metric $h_{0}$ on $M_0$. 

Write $(M,h)$ for the universal cover of $(M_{0},h_{0})$. Defining 
\[
\lambda_{1}(M,h) = \inf_{u\in C_{c}^{\infty}(M)\setminus\{0\}} \frac{\int_{M} |\nabla u|^{2}}{\int_{M} u^{2}},
\]
we claim that $\lambda_{1}(M,h) > 0$. This is a consequence of a result of Brooks \cite{Brooks} (see also \cite[\S 16]{KD}) which says that $\lambda_{1}(M,h) = 0$ if and only if $\pi_{1}(M_{0})$ is amenable. The definition of an amenable group can be found in \cite[\S 1]{Brooks} (among other places), but all that matters here is that $\pi_{1}(M_{0}) = F_{2}$, the free group on two generators, which is not amenable (see e.g., \cite[Proposition 4]{Tao:amenable}). 

We now fix $\eps = \frac{\lambda_{1}(M,h)}{6}$. 

For $\delta>0$ to be fixed below (depending only on $h$), define a family of metrics $h_t$ on $M_{0}$ by:
\[h_t = h_{0} + t^{2}(\Ric_{h_0}-2\eps h_{0}), \quad t\in (-\delta,\delta).\]
For $\delta$ sufficiently small, $h_t$ is a Riemannian metric. Finally, we define a metric $g$ on $X=M_{0}\times (-\delta,\delta)$ by
\[g=h_t + dt^2.\]

We make a few observations. First, since $\partial_t h_t|_{t=0} = 0$, the embedding $M_0\times\{0\}\subset (X,g)$ is totally geodesic. Moreover, by the Riccati equation (cf. \cite[Corollary 3.3]{GrayTubes}), we have that
\begin{equation}\label{eq.riccati}
	2\eps h_0 -  \Ric_{h_{0}} = - \tfrac 12 \partial_t^2 h_{t}|_{t=0} = R_g(\cdot,\partial_t,\partial_t, \cdot).
\end{equation}
Let $\{e_i\}_{i=1}^4$ be a local orthonormal frame on $M$, where $e_4=\partial_t$. By the Gauss equations for $M_0 \times \{0\}$ and \eqref{eq.riccati}, we have for $i\in \{1,2,3\}$
\[\Ric_g(e_i,e_i ) = \Ric_{h_{0}}(e_i,e_i) + R_g(e_i,e_4,e_4,e_i) = 2\eps.\]
Furthermore, by \eqref{eq.riccati}, we have
\[\Ric_g(e_4,e_4) = \sum_{i=1}^3 R_g(e_i,e_4,e_4,e_i) = 6\eps - R_{h_{0}} = 6\eps.\]
In particular, $\Ric_g\geq 2\eps$ along $M_{0}\times \{0\}$. If we take $\delta>0$ even smaller, we find that $\Ric_g\geq \eps$ on $X$.

Now we verify that the immersion $M^3\to (X^4,g)$ from the universal cover is stable. Note that the potential term in the stability operator for $M$ satisfies $|A_M|^{2} + \Ric_{g}(\nu,\nu) = 6\eps = \lambda_{1}(M)$. Thus, for any $u\in C^{\infty}_{c}(M)$, we have
\[
\int_{M} |\nabla u|^{2} \geq \lambda_{1}(M,h) \int_{M} u^{2}  = \int_{M}(|A_{M}|^{2}+\Ric_{g}(\nu,\nu))u^{2}.
\]
This completes the proof (after scaling so that $\Ric_{g}\geq 1$). Note that because the image of $M$ is an embedded submanifold, the sectional curvature is uniformly bounded below along a tubular neighborhood of $M$. 

\begin{remark}
It is an interesting question as to whether or not one can find $M^{3}\to (X^{4},g)$ stable minimal where $(X^4,g)$ is a \emph{closed} (or complete) manifold with $\Ric_{g} \geq 1$ at all points. It is tempting to try to study the metric of positive Ricci curvature on $(S^{2}\times S^{2}) \# (S^{2} \times S^{2})$ constructed in \cite{ShaYang}. If the construction is carried out in the most symmetric way possible, there will be a totally geodesic submanifold diffeomorphic to $ M_{0}:=(S^{1}\times S^{2}) \# (S^{1}\times S^{2})$. If one could modify the construction so as to keep positivity of the curvature, but to ensure that $0\leq \Ric_{g}(\nu,\nu) \to 0$ along $M_{0}$, then the proof used above would show that the universal cover of $M_{0}$ was eventually stable. 
\end{remark}

\subsection{Non-negative sectional curvature and strictly positive scalar curvature in six dimensions} \label{sec:six-dim}

Choose $(\mathbf{R}^4,g)$ as in Appendix \ref{sec:pos.K.cx} and consider $(X^6,g_X) : = (\mathbf{R}^4,g)\times (S^2,g_{S^2})$. Then, $(X^6,g_X)$ will have non-negative sectional curvature, scalar curvature $R_{g_X}\geq 2$, and strictly positive Ricci curvature. On the other hand, if we cross the totally geodesic $\mathbf{R}^3\to(\mathbf{R}^4,g)$ by $S^2$, we find $M^5 := \mathbf{R}^3\times S^2 \to (X^6,g_X)$ two-sided minimal immersion. The immersion $M\to (X,g_X)$ will be stable. To see this, consider $B_\rho(0)\times S^2\subset M$ for $\rho>0$. The first eigenfunction of the stability operator on this set will be $S^2$-invariant, so the argument used in Appendix \ref{sec:pos.K.cx} (the Hardy inequality) implies that the first eigenvalue of the stability operator on this set is positive. Letting $\rho\to \infty$, we see that $M$ is stable. On the other hand $\Ric_{g_X}(\nu,\nu)>0$ along $M$. 

\section{Pulling back immersions along a local diffeomorphism} \label{app:pullback}

Suppose that $X,Y,M$ are smooth manifolds, $\Psi : Y\to X$ is a local diffeomorphism and $F : M\to X$ is an immersion. We describe here how to ``pull back\footnote{This is a pullback in the category theoretic sense (in the category of smooth maps/manifolds). Note that the pullback of two maps need not always exist in this category, but it does when the maps are transversal. } $F$ along $\Psi$.'' This construction is presumably well-known.  

Below, we will use the standard notation $f \pitchfork Z$ to mean that the map $f$ is transverse to the submanifold $Z$. 

Consider the map 
\[
F \times \Psi : M\times Y \to X\times X.
\]
Write $\Delta = \{(x,x) : x\in X\}$ for the diagonal in $X\times X$. 
\begin{lemma}
$F \times \Psi\pitchfork \Delta$. 
\end{lemma}
\begin{proof}
We note that
\[
(F\times \Psi)^{-1}(\Delta) = \{(m,y) \in M\times Y : F(m)=\Psi(y)\}. 
\]
As such, for $(m,y) \in (F\times \Psi)^{-1}(\Delta)$, we find
\[
\Image d(F\times \Psi)_{(m,y)} =  \Image dF_{m} \times T_{\Psi(y)}X.
\]
Hence, for $(v_{1},v_{2}) \in T_{F(m)}X \times T_{\Psi(y)}X$, we can write
\[
(v_{1},v_{2}) = (v_{1},v_{1}) + (0,v_{2}-v_{1}) \in T_{(F(m),\Psi(y))}\Delta + \Image d(F\times \Psi)_{(m,y)}. 
\]
This completes the proof. 
\end{proof}
Thus $S : = (F\times \Psi)^{-1}(\Delta)$ is a submanifold of $M\times Y$. Recall that
\[
T_{s}S = (d(F\times \Psi)_{s})^{-1}(T_{(F(s),\Psi(s))}\Delta)
\]
and
\[
\dim M + \dim Y - \dim S = \codim(S\subset Y\times M) = \codim(\Delta\subset X\times X) = \dim X,
\]
so (because $\dim Y= \dim X$) we have
\[
\dim S = \dim M. 
\]
Write $F_{S} : S \to Y$ for the restriction of the projection map $M\times Y \to Y$ and similarly for $\Psi_{S} : S\to M$. In particular, the following diagram commutes
\[
\xymatrix{
S \ar[r]^{F_{S}} \ar[d]_{\Psi_{S}} & Y\ar[d]^{\Psi} \\
M \ar[r]_{F} & X
}
\]
We now check that the maps $ \Psi_{S},F_{S}$ have the desired properties. 
\begin{lemma}
The map $F_{S}:S\to Y$ is an immersion and the map $\Psi_{S} : S\to M$ is a local diffeomorphism. 
\end{lemma}
\begin{proof}
We start with $F_{S}$. For $s \in S$, write $s=(m,y) \in Y\times M$. We have that $(dF_{S})_{s}$ is the restriction of the projection onto the second factor map $\pi_{2} : T_{m}M\times T_{y}Y\to T_{y}Y$. As such,
\[
\ker (dF_{S})_{s} = (T_{m}M \times 0) \cap T_{s}S. 
\]
Hence, if $(t,0) \in \ker (dF_{S})_{s}$ then 
\[
(t,0) \in (d(F\times \Psi)_{s})^{-1}(T_{(F(s),\Psi(s))}\Delta),
\]
i.e., $dF_{m}(t) = 0$. Because $F$ is an immersion, we have $t=0$. This proves that $F_{S}$ is an immersion. For $\Psi_{S}$, note that $\dim S = \dim M$, so it suffices to prove that $\Psi_{S}$ is an immersion. The proof is then identical to the one we just gave for $F_{S}$. 
\end{proof}

\section{Existence of local covering maps with good regularity}\label{app:bd.geo}

Recall the definition of $Q$-weakly bounded geometry from Definition \ref{defi:bound.geo}. The following result is well-known (cf.\ \cite{RST}, \cite[Lemma 2.1]{NaberTianII}, and \cite[Exercise 11.6.15]{Petersen:Riemannian}). We sketch the proof below, referring to \cite{RST,Petersen:Riemannian} for several crucial details.
\begin{proposition}\label{prop:bd.geo}
	If $(X^{n},g)$ is a complete manifold with  $|\textnormal{sec}| \leq K < \infty$, then there is $Q=Q(K)$ so that $(X^n,g)$ has $Q$-weakly bounded geometry. 
\end{proposition}
\begin{proof}
	Fix $x \in X$ and choose an orthonormal basis for $T_{x}X$. We will identify $T_{x}X$ with $\R^{n}$ (so $g_{x}$ agrees with the standard inner product on $\R^{n}$). Using Jacobi field estimates (and $|\textnormal{sec}| \leq K < \infty$), there is $r_{0}=r_{0}(K,n)>0$ so that 
	\[
	\exp_{x} : B(0,4r_{0}) \subset \R^{n} \to X
	\]
	is a local diffeomorphism, and $\tilde g : = \exp_{x}^{*}g$ satisfies $\tfrac 12 \delta \leq \tilde g\leq 2\delta$ on $B(0,4r_{0})\subset \R^{n}$ as quadratic forms. By \cite[Lemma 2.2]{RST} it holds that $\inj_{\tilde g}(v) \geq i_0=i_0(K,n)>0$ for $v\in B(0,r_0)$. The assertion then follows by constructing harmonic coordinates for $\tilde g$ in a uniformly big neighborhood of $0$ as in \cite[Theorem 11.4.3]{Petersen:Riemannian} (see also \cite{Anderson} and \cite[Theorem 2.1]{RST})
\end{proof}

% REFERENCES
\bibliographystyle{amsalpha}
\bibliography{stable_min_dim_4}

\end{document}